\documentclass[a4paper,10pt]{amsart}

\usepackage{graphicx}
\usepackage{dirtytalk}
\usepackage{amsthm}
\usepackage{amssymb}
\usepackage{amsmath}
\usepackage{accents}
\usepackage{hyperref}
\usepackage{xcolor}
\usepackage{enumerate}
\usepackage{listings}
\usepackage{complexity}
\usepackage{blkarray}
\usepackage{braket}
\usepackage{lmodern}
\usepackage[english]{babel}
\usepackage{enumitem}

\usepackage[font=small,labelfont=bf]{caption}
\usepackage[font=small,labelfont=normalfont,labelformat=simple]{subcaption}
\graphicspath{{figs/}}

\newtheorem{theorem}{Theorem}

\newtheorem{lemma}[theorem]{Lemma}

\newtheorem{observation}{Observation}

\author
{
Anders Martinsson
}

\author
{
Raphael Steiner 
}
\thanks{Department of Computer Science, Institute of Theoretical Computer Science, ETH Z\"{u}rich, Switzerland.  \texttt{\{anders.martinsson,raphaelmario.steiner\}@inf.ethz.ch}. Research of R.S. funded by SNSF Ambizione grant No. 216071.}

\date{\today}

\title{Vertex-critical graphs far from edge-criticality}

\begin{document}
\maketitle

\begin{abstract}
Let $r$ be any positive integer. We prove that for every sufficiently large $k$ there exists a $k$-chromatic vertex-critical graph $G$ such that $\chi(G-R)=k$ for every set $R \subseteq E(G)$ with $|R|\le r$. 
This partially solves a problem posed by Erd\H{o}s in 1985, who asked whether the above statement holds for $k \ge 4$.
\end{abstract}

\section{Introduction}

The chromatic number $\chi(G)$ of a graph $G$ is among the oldest and most fundamental graph parameters, but despite its intensive study by researchers across the field for more than a century, many fundamental open problems remain. In many instances, we would like to show that for some number $k$, all graphs in an infinite class $\mathcal{G}$ of graphs have chromatic number less than $k$. Often times, the graph class $\mathcal{G}$ at hand will also have the property that it is closed under taking induced, or even arbitrary, subgraphs. In this case, a central idea for bounding the chromatic number is to consider the \emph{minimal} graphs in $\mathcal{G}$ with chromatic number $k$. These graphs have the special property that removing any vertex (if $\mathcal{G}$ is closed under induced subgraphs) or any edge (if $G$ is closed under subgraphs) reduces the chromatic number from $k$ to $k-1$. This enforces many constraints on such minimal graphs, for instance sufficiently high minimum degree and edge-connectivity, among others. Such properties can then prove useful when showing the non-existence of minimal $k$-chromatic graphs in $\mathcal{G}$, which in turn establishes that the chromatic number of graphs in $\mathcal{G}$ is less than $k$. 

Because of this and many other applications, the notion of \emph{color-critical graphs} has emerged. Given an integer $k$, a graph $G$ is called \emph{$k$-chromatic vertex-critical} if $\chi(G)=k$, but $\chi(G-v)=k-1$ for every $v \in V(G)$. Similarly, it is called \emph{$k$-chromatic edge-critical}, if $\chi(G)=k$ but $\chi(G-e)=k$ for every $e \in E(G)$. Note that edge-criticality implies vertex-criticality if we exclude redundant cases in which $G$ has isolated vertices. 

A considerable amount of effort has been put into understanding how different the notions of vertex-criticality and edge-criticality can be. Already in 1970, G.~Dirac~\cite{lattanzio} conjectured that for every integer $k \ge 4$, there exists a $k$-chromatic vertex-critical graph $G$ which at the same time is very much not edge-critical, in the sense that the deletion of any single edge does \emph{not} lower its chromatic number. In the following, let us say that such a graph \emph{has no critical edges}. Dirac's problem for a long time remained poorly understood. It was not before 1992 that Brown~\cite{brown} finally found a first construction of some vertex-critical graph with no critical edges, in fact, he found such a construction for $k=5$. Later, in 2002, Lattanzio~\cite{lattanzio} found a more general construction which proved Dirac's conjecture for every integer $k \ge 5$ such that $k-1$ is not a prime number. Shortly after, Jensen~\cite{jensen} provided a construction of $k$-chromatic vertex-critical graphs with no critical edges for every $k \ge 5$. This leaves only the case $k=4$ of Dirac's conjecture open today, which remains an intriguing open problem.
A wide-ranging strengthening of Dirac's conjecture was proposed by Erd\H{o}s in 1985~\cite{erdos}, as follows. 

\medskip
\say{\textit{I recently heard from Toft the following conjecture of Dirac: Is it true that for every $k \ge 4$ there is a $k$-chromatic vertex-critical graph which remains $k$-chromatic if any of its edges is omitted. If the answer as expected is yes, then one could ask whether it is true that for every $k \ge 4$ and $r$ there is a vertex-critical $k$-chromatic graph which remains $k$-chromatic if any $r$ of its edges are omitted.}} 
\begin{center}(Paul Erd\H{o}s, 1985, top of page~113 in \cite{erdos})\end{center}

This problem is also mentioned in several other sources, for instance it is listed as Problem~5.14 in the book~\cite{jensentoft} by Jensen and Toft and on page 66 in Chapter 4 of the Erd\H{o}s open problem collection by Chung and Graham~\cite{chungraham}, see also the online version of the problem~\cite{chungrahamonline}. 

The question of Erd\H{o}s can be rephrased as asking whether for arbitrarily large numbers $r$ there exist $k$-chromatic vertex-critical graphs for $k \ge 4$ that are ``pretty far'' from any of their $(k-1)$-chromatic spanning subgraphs, in the sense that one has to remove more than $r$ edges to reach any such subgraph. As described above, the case $r=1$ of this problem is well-understood, however, not much seems to be known beyond that, when $r\ge 2$.
\medskip
\paragraph{\textbf{Our contribution.}} In this paper, we resolve the problem by Erd\H{o}s for any value $r$ and all sufficiently large values $k$. To the best of our knowledge, these are the first known examples of such graphs for arbitrarily large values of $r$.

\begin{theorem}\label{thm:main}
For every $r \in \mathbb{N}$ there is some $k_0\in \mathbb{N}$ such that for every $k \ge k_0$ there exists a $k$-chromatic vertex-critical graph $G$ such that $\chi(G-R)=k$ for every $R \subseteq E(G)$ with $|R|\le r$.
\end{theorem}

Our result still leaves open Erd\H{o}s' question when $k\ge 4$ is fixed as a small value and $r$ tends to infinity, and this remains an interesting open case of the problem. The rest of this note is devoted to presenting our proof of Theorem~\ref{thm:main}. The main idea of the construction is to use the existence of uniform hypergraphs that admit a perfect matching upon the removal of any single vertex, but at the same time are locally rather sparse. Such hypergraphs in turn can be constructed randomly, using the recent advances on Shamir's hypergraph matching problem.
\newline\noindent 
\paragraph{\textbf{Notation.}}  For a graph $G$ and a subset $X\subseteq V(G)$ of its vertices, $G[X]$ denotes the subgraph of $G$ induced by $X$. A \emph{hypergraph} is a tuple $(V,E)$ where $V$ is a finite set and $E \subseteq 2^V\setminus \{\emptyset\}$. Given a hypergraph $H=(V,E)$, we denote by $V(H)=V$ its vertex- and by $E(H)=E$ its hyperedge-set. For $v \in V(H)$, we denote by $H-v$ the hypergraph with vertex-set $V(H)\setminus \{v\}$ and hyperedge-set $\{e\in E(H)|v \notin e\}$. For $e \in E(H)$, $H-e:=(V(H),E(H)\setminus \{e\})$ is the hypergraph obtained by omitting~$e$. Given a hypergraph $H$, its \emph{$2$-section} is the graph $G_2^H$ on the same vertex-set $V$ and where $uv \in E(G_2^H)$ if and only if there is some $e \in E(H)$ with $u,v \in e$. 

\section{Proof of Theorem~\ref{thm:main}}

In the following, given positive integers $n,k$ and a probability value $p \in [0,1]$, we denote by $\mathcal{H}_s(n,p)$ the binomial $s$-uniform random hypergraph on vertex-set $V=[n]=\{1,\ldots,n\}$, obtained by including every $s$-subset of $V$ as a hyperedge independently with probability $p$. Given a hypergraph $H$, a \emph{perfect matching} of $H$ is a collection $\{e_1,\ldots,e_t\}\subseteq E(H)$ of hyperedges that form a set-partition of $V(H)$. Note that if $H$ is an $s$-uniform hypergraph, then the existence of a perfect matching necessitates $|V(H)|\equiv 0 \text{ }(\text{mod }s)$. One of the most famous problems in probabilistic graph theory for a long  time was \emph{Shamir's problem}, that asked to determine the threshold for the random hypergraph $\mathcal{H}_s(n,p)$ with $n \equiv 0 \text{ }(\text{mod }s)$ to contain a perfect matching. This threshold was determined up to a multiplicative error in a breakthrough-result by Johannson, Kahn and Vu~\cite{johansson} in 2008, as follows.
\begin{theorem}[cf.~\cite{johansson}]\label{thm:johansson}
    For every integer $s \ge 1$ there exists a constant $C=C(s)>0$ such that with $p=p(n)=\frac{C\log n}{n^{s-1}}$ it holds that $\mathcal{H}_s(n,p)$ has a perfect matching w.h.p. provided that $n \equiv 0 \text{ }(\text{mod }s)$.
\end{theorem}
We remark that recently, Kahn~\cite{kahn} has determined the threshold in Shamir's problem even more precisely, showing that taking $C=(1+o(1))(s-1)!$ is sufficient (and best-possible). We now use this probabilistic result to deduce the existence of uniform hypergraphs with special properties, as follows.

\begin{lemma}\label{lemma:randomhyp}
Let $s \ge 2, m \ge 1$ be fixed integers. Then for every sufficiently large integer $n$ such that $n \equiv 1 \text{ }(\text{mod }s)$, there exists an $s$-uniform hypergraph $H$ on $n$ vertices with the following properties.
\begin{enumerate}[label=(\roman*)]
    \item For every $v \in V(H)$, the hypergraph $H-v$ admits a perfect matching.
    \item For every set $F\subseteq E(H)$ of hyperedges with $|F|\le m$, we have $$\left|\bigcup_{e \in F}{e}\right|\ge (s-1)|F|.$$
\end{enumerate}
\end{lemma}
\begin{proof}
Let $p(n):=\frac{C\log n}{n^{s-1}}$ be as in the statement of Theorem~\ref{thm:johansson}. Then, for every $n\equiv 1 \text{ }(\text{mod }s)$ chosen large enough, by Theorem~\ref{thm:johansson} we have 
$$\mathbb{P}(\mathcal{H}_s(n-1,p(n-1))\text{ has a perfect matching})\ge \frac{1}{2}.$$ Now, define $q(n):=\lceil2\log_2(n)\rceil p(n-1)=\Theta\left(\frac{\log^2 n}{n^{s-1}}\right)$. In the following, we show that $\mathcal{H}_s(n,q(n))$ satisfies both (i) and (ii) w.h.p. provided $n\equiv 1 \text{ }(\text{mod }s)$, which will then imply the statement of the lemma.

Imagine sampling a random $s$-uniform hypergraph $\tilde{H}$ on vertex-set $[n]$ as the union of $l:=\lceil 2\log_2(n)\rceil$ independently generated instances of $\mathcal{H}_s(n,p(n-1))$, which we call $H_1,\ldots,H_l$. Note that the distribution of the random hypergraph $\tilde{H}=H_1\cup \dots \cup H_l$ follows that of a binomial random hypergraph $\mathcal{H}_s(n,q'(n))$ with edge-probability $q'(n)=1-(1-p(n-1))^l=1-(1-p(n-1))^{\lceil2\log_2(n)\rceil}\le q(n)$. Now fix a vertex $v \in [n]$. From the above we have, since the property of having a perfect matching is monotone,
\begin{align*}
 &\mathbb{P}(\mathcal{H}_s(n,q(n))-v \text{ has no perfect matching}) \\  \le \: & \mathbb{P}(\mathcal{H}_s(n,q'(n))-v \text{ has no perfect matching}) \\ = \: & \mathbb{P}(\tilde{H}-v \text{ has no perfect matching})\\ \le \: & \prod_{i=1}^{l}{\mathbb{P}(H_i-v\text{ has no perfect matching})}.   
\end{align*}
Since for every $i$ the distribution of $H_i-v$ follows that of an $\mathcal{H}_s(n-1,p(n-1))$, from the above we have that $\mathbb{P}(H_i-v\text{ has no perfect matching})\le \frac{1}{2}$ for $i=1,\ldots,l$. Altogether, it follows that
$$\mathbb{P}(\mathcal{H}_s(n,q(n))-v \text{ has no perfect matching})\le \left(\frac{1}{2}\right)^{2\log_2(n)}=\frac{1}{n^2}.$$ Using a union bound over all choices of $v$, this implies that 
$$\mathbb{P}\left(\bigcup_{v \in [n]}\{\mathcal{H}_s(n,q(n))-v \text{ has no perfect matching}\}\right)\le\frac{n}{n^2}=\frac{1}{n}.$$ Thus, w.h.p. $\mathcal{H}_s(n,q(n))$ satisfies property (i). 

Let us now move on to property (ii). For that purpose, we want to show that w.h.p. for every number $f=1,\ldots,m$, no subset of $[n]$ of size $(s-1)f-1$ contains $f$ hyperedges from $\mathcal{H}_s(n,q(n))$. Let $T(s,f)$ denote the number of labelled hypergraphs on $(s-1)f-1$ vertices containing $f$ hyperedges. Using a simple union bound over all choices of subsets of $[n]$ of size $(s-1)f-1$ and the possible configurations of edges on those subsets, we obtain that the probability that there exist $f$ hyperedges in $\mathcal{H}_s(n,q(n))$ spanning less than $(s-1)f$ vertices is at most
$$\binom{n}{(s-1)f-1}\cdot T(s,f) \cdot q(n)^f=O\left(n^{(s-1)f-1}\cdot \left(\frac{\log^2 n}{n^{s-1}}\right)^f\right)=O\left(\frac{\log^{2f}n}{n}\right).$$ Thus, w.h.p. we have that $\mathcal{H}_s(n,q(n))$ also satisfies item (ii) of the lemma. This concludes the proof. 
\end{proof}

Next, would like to use the hypergraphs from the previous lemma to construct graphs that satisfy the conditions of Theorem~\ref{thm:main}. To do so, we first need to prove a technical result about the number of edges that can be spanned by any $(s+1)$-subset of vertices in the $2$-section of these hypergraphs, namely Lemma~\ref{lemma:edgebound}. To prove Lemma~\ref{lemma:edgebound}, we first establish an auxiliary result on hypergraphs in the form of Lemma~\ref{lemma:blocks}, which in turn needs the following elementary but important observation.
\begin{observation}\label{obs:connectedbound}
Let $H=(V,E)$ be a connected hypergraph (that is, $G_2^H$ is connected). Then 
$$|V|\le 1+\sum_{e\in E}{(|e|-1)}.$$
\end{observation}
\begin{proof}
Let $T$ be a spanning tree of $G_2^H$. For every edge $t \in E(T)$, assign a hyperedge $e(t) \in E$ such that $t \subseteq e(t)$. For each $e \in E$, let $T_e\subseteq T$ be the forest induced by the edges $\{t \in E(T)|e(t)=e\}$. Clearly, $V(T_e)\subseteq e$ for every $e \in E$, and thus
$$|V|-1=|E(T)|=\sum_{e \in E}{|E(T_e)|}\le \sum_{e\in E}{\max\{0,|V(T_e)|-1\}}\le \sum_{e \in E}{(|e|-1)},$$ as desired.
\end{proof}
\begin{lemma}\label{lemma:blocks}
Let $H=(V,E)$ be a hypergraph with $|V|\ge 4$ and $V \notin E$. Suppose further that
for every set $F \subseteq E$ of hyperedges, we have 
$$\left|\bigcup_{e\in F}{e}\right|\ge \sum_{e\in F}{(|e|-1)}.$$
Then there exists a set $W\subseteq V$ of size at most $2$ such that $G_2^H-W$ is disconnected. 
\end{lemma}
\begin{proof}
Suppose first that there exists at least one hyperedge $e_0 \in E$ with $|e_0|\ge 3$. By assumption, $V \notin E$, and thus there exists some vertex $v \in V\setminus e_0$. Let us now consider the graph $G=G_2^{H-e_0}$, the $2$-section of the hypergraph $H-e_0$ obtained from $H$ by deleting $e_0$. Let $C$ be the vertex-set of the unique connected component of $G$ that contains $v$. We claim that $|C \cap e_0|\le 2$. To that end, define $F$ as the set of hyperedges of $H$ that are contained in $C$. Clearly, $e_0\notin F$, since $v \in C$ and $v\notin e_0$. Note that, since every hyperedge $e\in E\setminus\{e_0\}$ induces a clique in $G$, we have that $\bigcup_{e\in F}{e}=C$ and that the hypergraph $H'=(C,F)$ is connected. These facts imply via Observation~\ref{obs:connectedbound} that
$$\left|\bigcup_{e\in F}{e}\right|=|C|\le 1+\sum_{e \in F}{(|e|-1)}.$$
On the other hand, by applying the assumption of the lemma to the edge-set $F \cup \{e_0\}$, we find 
$$\sum_{e \in F\cup \{e_0\}}{(|e|-1)}\le \left|e_0 \cup \bigcup_{e\in F}{e}\right|=|e_0 \cup C|=|e_0|+|C|-|e_0 \cap C|.$$
Subtracting $(|e_0|-1)$ from both sides yields 
$$\sum_{e\in F}{(|e|-1)}\le |C|+1-|e_0 \cap C|.$$
Plugging the above into the first inequality we get
$|C|\le |C|+2-|e_0 \cap C|,$ and thus $|e_0 \cap C|\le 2$, as claimed. We now set $W:=e_0\cap C$ and claim that $G_2^H-W$ is disconnected. Indeed, it follows readily from the definition of $C$ that no edge in $G_2^H-W$ connects a vertex in $C\setminus W=C\setminus e_0$ to a vertex in $V\setminus C$. Further, since $v \in C\setminus e_0$ we have that the first set is non-empty, and since $|V\setminus C|\ge |e_0\setminus C|=|e_0|-|e_0\cap C|\ge 3-2=1>0$, the second set is also non-empty. Thus, $G_2^H-W$ is indeed disconnected, which concludes the proof in this case.

For the second case, assume that $|e|\le 2$ for every $e \in E$. W.l.o.g. (since they do not have an effect on $G_2^H$) we may assume that $H$ contains no hyperedges of size $1$, i.e., $H$ is a graph and $G_2^H=H$. If $H$ has a vertex of degree at most $1$, then the statement of the lemma trivially holds, so suppose that $H$ has minimum degree at least $2$. The condition of the lemma now yields $|E|=\sum_{e\in E}{(|e|-1)}\le |\bigcup_{e\in E}{e}|\le |V|$. This directly implies via the handshake-lemma that $H$ is a $2$-regular graph. It is trivial to see that every such graph on at least $4$ vertices contains a cut-set $W$ consisting of at most $2$ vertices, and this concludes the proof. 
\end{proof}

\begin{lemma}\label{lemma:edgebound}
Let $s \ge 3$ be an integer, let $H$ be an $s$-uniform hypergraph such that $\left|\bigcup_{e \in F}{e}\right| \ge (s-1)|F|$ holds for all $F \subseteq E(H)$ with $|F|< 2^{s+1}$. Let $G$ be the $2$-section of $H$. Then, for every set $X \subseteq V(G)$ of size $s+1$, it holds that $|E(G[X])|\le \binom{s}{2}+2$. 
\end{lemma}
\begin{proof}
Let $H_X$ denote the hypergraph obtained by \emph{restricting} $H$ to $X$, that is, $V(H_X):=X$ and $E(H_X)=\{e \cap X|e \in E(H), |e\cap X|\ge 2\}$. Note that the $2$-section of $H_X$ equals $G[X]$. Further note that for every subset $F \subseteq E(H)$ of size less than $2^{s+1}$, it holds that 
$$\left|\bigcup_{e \in F}{(e\cap X)}\right|\ge \left|\bigcup_{e \in F}{e}\right|-\sum_{e \in F}{|e\setminus X|}$$ $$\ge (s-1)|F|-\sum_{e \in F}{|e\setminus X|}=\sum_{e \in F}{(s-1-|e\setminus X|)}=\sum_{e \in F}{(|e\cap X|-1)}.$$
This directly implies that $\left|\bigcup_{e \in F}{e}\right|\ge \sum_{e \in F}{(|e|-1)}$ for every subset $F \subseteq E(H_X)$.
We can therefore apply Lemma~\ref{lemma:blocks}, which implies that there exists a set $W\subseteq X$ of size at most $2$ such that $G[X]-W$ is disconnected. Thus, there exist disjoint non-empty sets $A, B$ such that $A\cup B=X\setminus W$ and no edge in $G[X]$ connects $A$ and $B$. Note that as $|A|,|B|\ge 1$ and $|A|+|B|=|X|-|W|\ge (s+1)-2=s-1$, we have $|A||B|\ge s-2$. We conclude that 
$$|E(G[X])|\le \binom{s+1}{2}-|A||B|\le \binom{s+1}{2}-(s-2)=\binom{s}{2}+2.$$ This concludes the proof.
\end{proof}

\begin{proof}[Proof of Theorem~\ref{thm:main}]
Let an integer $r \ge 1$ be given. Define $s:=r+3$ and $m:=2^{s+1}$. By Lemma~\ref{lemma:randomhyp} there exists some $n_0 \in \mathbb{N}$ such that for every integer $n \ge n_0$ with $n \equiv 1 \text{ }(\text{mod }s)$, there exists an $s$-uniform hypergraph $H$ on $n$ vertices with the following properties.
\begin{itemize}
    \item For every $v \in V(H)$, the hypergraph $H-v$ admits a perfect matching.
    \item For every set $F\subseteq E(H)$ of hyperedges with $|F|\le m=2^{s+1}$, we have $$\left|\bigcup_{e \in F}{e}\right|\ge (s-1)|F|.$$
\end{itemize}
Define $k_0:=\lceil\frac{n_0-1}{s}\rceil+1$ and let $k \ge k_0$ be any given integer. Let $H$ be an $s$-uniform hypergraph on $n:=s(k-1)+1\ge n_0$ vertices satisfying the properties above. Finally, we define a graph $G$ as the complement of the $2$-section $G_2^H$ of $H$. We claim that it satisfies the properties required by the theorem, that is,

\begin{itemize}
    \item $G-v$ is $(k-1)$-colorable for every $v \in V(G)$, and
    \item for every set $R \subseteq E(G)$ of edges with $|R|\le r$, we have $\chi(G-R)\ge k$. 
\end{itemize}

To verify the first statement, consider any vertex $v$ and a perfect matching of $H-v$. Since $H$ is $s$-uniform, the perfect matching forms a partition of $V(H)\setminus\{v\}=V(G)\setminus\{v\}$ into $\frac{n-1}{s}=k-1$ sets, each inducing a hyperedge in $H$ and thus an independent set in $G$. Hence we have $\chi(G-v)\le k-1$.

Now let $R \subseteq E(G)$ with $|R|\le r$ be given. We claim that $\alpha(G-R)\le s$, i.e., that there exists no independent set in $G-R$ of size $s+1$, which will then imply $\chi(G-R)\ge \frac{n}{\alpha(G-R)}\ge \frac{n}{s}>k-1$, as desired. Suppose towards a contradiction that there is some $X\subseteq V(G)$ of size $s+1$ that is independent in $G-R$. Then $G[X]$ contains at most $r$ edges, and thus its complement graph, namely $G_2^H[X]$, contains at least $\binom{s+1}{2}-r=\binom{s}{2}+s-r=\binom{s}{2}+3$ edges. However, by Lemma~\ref{lemma:edgebound} applied to $H$, we find that $|E(G_2^H[X])|\le \binom{s}{2}+2$, a contradiction. This shows that indeed, $\alpha(G-R)\le s$ for every $R \subseteq E(G)$ with $|R|\le r$, concluding the proof.  
\end{proof}

\end{document}